\documentclass[12pt]{article}
\usepackage{amsmath,amsfonts,amssymb,amsthm}
\newcommand{\R}{\mathbb{R}}
\newcommand{\N}{\mathbb{N}}
\newcommand{\Z}{\mathbb{Z}}
\newcommand{\Q}{\mathbb{Q}}

\newcommand\bN{\mathbf{N}}
\newcommand{\ca}{\mathcal A}
\newcommand{\twosum}[2]{\sum_{\substack{#1\\#2}}}
\newcommand{\twoprod}[2]{\prod_{\substack{#1\\#2}}}
\newcommand{\p}{\mathcal P}
\newcommand{\cR}{\mathcal R}
\newtheorem{thm}{Theorem}[section]
\newtheorem{lem}[thm]{Lemma}
\begin{document}
\title{Cubic polynomials represented by norm forms}
\author{A.J. Irving\\
Mathematical Institute, Oxford}
\date{}

\maketitle

\begin{abstract}
We show that for an irreducible cubic $f\in\Z[x]$ and a full norm form $\bN(x_1,\ldots,x_k)$ for a number field $K/\Q$ satisfying certain hypotheses the variety 
$$f(t)=\bN(x_1,\ldots,x_k)\ne 0$$
satisfies the Hasse principle.  Our proof uses sieve methods.
\end{abstract}


\section{Introduction}

In this paper we will consider the Diophantine equation 
\begin{equation}
  \label{eq:diophantine}
  f(t)=\bN(x_1,\ldots,x_k)\ne 0,\\
\end{equation}
where $f\in\Z[x]$ is a polynomial and $\bN(x_1,\ldots,x_k)$ is a full norm form for some number field $K/\Q$.  Thus, for some basis $\omega_1,\ldots,\omega_k$  for  the degree $k$ extension $K/\Q$, we have 
$$\bN(x_1,\ldots,x_k)=N_{K/\Q}(x_1\omega_1+\ldots+x_k\omega_k).$$
We are interested in describing families of fields $K/\Q$ and polynomials $f$ for which (\ref{eq:diophantine}) satisfies the Hasse principle.  This means that if there is a solution in $\Q_p^{k+1}$, for all $p$, and in $\R^{k+1}$ then there is one in $\Q^{k+1}$.

Browning and Heath-Brown, in \cite{rhbbrowning}, describe many of the existing results on this problem.  They establish that the Hasse principle holds when $f$ is an irreducible polynomial of degree $2$ and $K/\Q$ is a quartic extension containing a root of $f$.  Their results were extended by Derenthal, 
Smeets and Wei, \cite{dsw}, who establish that for any quadratic $f$ and any extension $K/\Q$ the Brauer--Manin obstruction 
is the only obstruction to the Hasse principle.

We are interested in the case that $f$ is an irreducible cubic.  Previous work establishes that, when $K/\Q$ has degree $2$ or $3$, the only obstruction to the Hasse principle is the Brauer--Manin 
obstruction.  Specifically, if $[K:\Q]=2$ then (\ref{eq:diophantine}) defines a Ch\^atelet surface so the result follows by the work of Colliot-Th\'el\`ene, Sansuc and Swinnerton-Dyer \cite{crelle}, whereas if $[K:\Q]=3$ it follows from Colliot-Th\'el\`ene and Salberger \cite{ct-salb}.   As far as we know, no case of the Hasse principle has been established when $f$ is an irreducible cubic and $[K:\Q]>3$.  

If, instead of being irreducible, $f$ splits completely into linear factors over $\Q$ then the problem is considerably different.  A recent result of Browning and Matthiesen \cite{browningmat} establishes that for any such $f$ and any number field $K/\Q$ the Hasse principle holds provided that the Brauer--Manin obstruction is empty.

Our main result is the following, which establishes the Hasse principle for a certain class of fields $K/\Q$, whose degree may be arbitrarily large.

\begin{thm}\label{mainthm}
Let $f\in\Z[x]$ be an irreducible cubic   and let $K/\Q$ be a number field satisfying the Hasse norm principle.  Suppose that there exists a prime $q\ge 7$ such that for all but finitely many unramified primes $p$ with 
$p\not\equiv 1\pmod q$
the prime ideal factorisation of $p$, 
$$(p)=\prod_{i=1}^r P_i,$$
consists of prime ideals $P_i$ of norms $p^{f_i}$ with $f_1,\ldots,f_r$ coprime.  In addition, assume that the number field generated by $f$ is not contained in the cyclotomic field $\Q(\zeta_q)$.  Then (\ref{eq:diophantine}) satisfies the Hasse principle.
\end{thm}

An example of a field $K/\Q$ satisfying all the conditions of this theorem can be found by adjoining to $\Q$ a root of 
$$x^q-2,$$
for any prime $q\geq 7$.  
Since $[K:\Q]=q$ is prime, the extension $K/\Q$ satisfies the Hasse norm principle by the work of Bartels \cite{bartels-81}. In addition, for any prime $p\not\equiv 1\pmod q$ the equation 
$$x^q-2\equiv 0\pmod p$$ 
has a root.  If we exclude finitely many values of $p$ it then follows that $K$ has an integral ideal of norm $p$.  This clearly implies that the degrees $f_i$ are coprime.  In conclusion, for this choice of $K$ and any $f$ which does not generate a subfield of $\Q(\zeta_q)$, (\ref{eq:diophantine}) satisfies the Hasse principle.

After various algebraic reductions we will prove Theorem \ref{mainthm} using sieve methods.  We will show that for an integer to be a norm from $K/\Q$ it is sufficient that it satisfies certain congruences and that it has no prime factors $p\equiv 1\pmod q$.  We may therefore estimate the number of norms in a set of integers by sieving out these primes.  Our sieve problem will have dimension $\frac{2}{q-1}$.  For large $q$ this is close to $0$ and therefore both the upper and lower bounds coming from the sieve are close to the truth.  We will show that for $q\geq 7$ the losses in the sieve are sufficiently small to give us a positive lower bound for the number of rational points on (\ref{eq:diophantine}).  There are many well-known applications of sieves whose dimension is an integer or $\frac12$.  However we are not aware of any existing work which uses a sieve whose dimension is between $0$ and $\frac12$.

It seems very likely that the method of this paper can be adapted to prove weak approximation for the variety (\ref{eq:diophantine}) provided that it can be shown that weak approximation holds for the equation 
$$N_{K/\Q}(u)=1.$$
As part of our proof we will show that $p$-adic conditions, for finitely many primes $p$, can be imposed on the variable $t$.  To handle the infinite place our sieve would have to be modified: sieving a more general region instead of $[0,N]^2$.  This modification would enable us to find a rational solution, $(t,x_1,\ldots,x_k)$, to (\ref{eq:diophantine}) with the variable $t$ sufficiently close to any idelic point.  If we define $x\in K$ by 
$$x=x_1\omega_1+\ldots+x_k\omega_k.$$
then, for any $u\in K$ with $N_{K/\Q}(u)=1$, we can write 
$$ux=y_1\omega_1+\ldots+y_k\omega_k$$
and we have 
$$f(t)=\bN(y_1,\ldots,y_k).$$
It could be shown, using our assumption that weak approximation holds for $N_{K/\Q}(u)=1$, that we can choose a $u$ to make $(y_1,\ldots,y_k)$ sufficiently close to any idelic point.  

We decided to restrict to the case that $q$ is prime to simplify some of the technical details in the sieve.  It seems probable that the argument could work for composite $q$,  however the condition $q\geq 7$ would have to be changed as our bounds would involve the value of $\varphi(q)$.  We use the assumption that $f$ does not generate a subfield of $\Q(\zeta_q)$ to avoid any correlation between the splitting of primes in the number field $K/\Q$ and in the field generated by $f$.  This will be made precise in Lemma \ref{hecke}.  Observe that if $q\equiv 2\pmod 3$ then this condition is satisfied for all cubics $f$ as $\Q(\zeta_q)$ cannot contain a subfield of degree $3$.  

\subsection*{Acknowledgements}

This work was completed as part of my DPhil, for which I was funded by EPSRC grant EP/P505666/1. I am very grateful to the
EPSRC for funding me and to my supervisor, Roger Heath-Brown, for all
his valuable help and advice.

\section{Algebraic Reduction of the Problem}

It does not matter which norm form $\bN$ we choose as they are all equivalent under a linear change of variables defined over $\Q$.  In particular we may assume that $\bN\in \Z[x_1,\ldots,x_k]$.   As we eventually wish to apply sieve methods we reduce from a problem over $\Q$ to one over $\Z$.  We therefore let $f(a,b)$ denote the homogeneous form of $f$:  
$$f(a,b)=b^3f(\frac{a}{b}).$$

\begin{lem}\label{homogeneous}
Suppose there exist integers $a$ and $b$ for which 
$$b,f(a,b)\in N_{K/\Q}(K^*).$$
There is then a solution to (\ref{eq:diophantine}) over $\Q$.
\end{lem}
\begin{proof}
Clearly $b\ne 0$.  We have 
$$f(\frac{a}{b})=b^{-3}f(a,b).$$
This is a norm from $K$ since both $b$ and $f(a,b)$ are, and the norm map is multiplicative.  In addition $f(\frac{a}{b})\ne 0$ since $f$ is irreducible.
\end{proof}

We know that the Hasse norm principle holds for $K/\Q$.  This means that a nonzero $a\in\Q$ is a norm from $K$ if and only if it is a norm from the group of ideles $I_K$:
$$\Q^*\cap N_{K/\Q}(I_K)=N_{K/\Q}(K^*).$$ 
Let $V$ denote the set of all places of $K$.  To show that $a\in\Q^*$ is a norm from $K$ it is thus sufficient to construct a sequence $(x_v)_{v\in V}$, where $x_v$ is a nonzero element of $K_v$,  with the following two properties:

\begin{enumerate}
\item For all non-Archimedean places $v$, with finitely many exceptions, we have $x_v\in R_v^*$, where $R_v$ is the valuation ring of $K_v$.  This condition ensures that $(x_v)\in I_K$.

\item For all places $w$ of $\Q$ we have 
$$\prod_{v|w}N_{K_v/\Q_w}(x_v)=a.$$
\end{enumerate}

We will derive arithmetic conditions which are sufficient to show that a nonzero integer $a$ is in $N_{K/\Q}(K^*)$.

\begin{lem}\label{unramified}
Suppose $a\ne 0$ is an integer.  Let $p$ be a prime which does not divide $a$ and which is unramified in $K/\Q$. Then there exist elements $x_v\in R_v^*$, for each place $v$ above $p$, such that 
$$\prod_{v|p}N_{K_v/\Q_p}(x_v)=a.$$
\end{lem}
 \begin{proof}
Let $v_1$ be one of the places above $p$.  Since $p$ is unramified in $K/\Q$ we know that the extension $K_{v_1}/\Q_p$ is unramified.  Furthermore, $p\nmid a$ so $a\in \Z_p^*$.  It follows by local class field theory, (for example Gras \cite[Corollary 1.4.3, part
(ii), page 75]{gras}), that there exists $x_{v_1}\in K_{v_1}$ with 
$$N_{K_{v_1}/\Q_p}(x_{v_1})=a.$$
We must have $x_{v_1}\in R_{v_1}^*$ since 
$$|x_{v_1}|=|N_{K_{v_1}/\Q_p}(x_{v_1})|^{1/[K_{v_1}:\Q_p]}=|a|^{1/[K_{v_1}:\Q_p]}=1.$$
For all $v|p$ with $v\ne v_1$ we define $x_v=1$ so 
$$N_{K_v/\Q_p}(x_v)=1.$$
The result follows.
\end{proof}

For any fixed $a$ this lemma has dealt with all but a finite number of places.  It follows that for the remaining places we need not consider the condition $x_v\in R_v^*$.

\begin{lem}\label{normp}
Suppose $a\ne 0$ is an integer.  Let $p$ be a prime dividing $a$ which is unramified in $K/\Q$.  In addition suppose that in the prime ideal factorisation 
$$(p)=\prod P_i,$$
with $N(P_i)=p^{f_i}$, the various $f_i$ are coprime.
 It follows that there exist $x_v\in K_v^*$, for each $v|p$, with 
$$\prod_{v|p}N_{K_v/\Q_p}(x_v)=a.$$
\end{lem}
\begin{proof}
Let $v_i$ be the place corresponding to the prime ideal $P_i$ in the factorisation of $(p)$.  We have $[K_{v_i}:\Q_p]=f_i$ so 
$$N_{K_{v_i}/\Q_p}(a)=a^{f_i}.$$  
Since the $f_i$ are coprime there exist integers $k_i$ such that 
$$\sum k_if_i=1.$$
The result follows on taking $x_{v_i}=a^{k_i}$.
\end{proof}

It remains to deal with the primes $p$ which ramify in $K/\Q$. For such primes it is easier to interpret the idelic condition in terms of the solubility of the norm equation over $\Q_p$.

\begin{lem}\label{equivalentlocal}
Let $a\ne 0$ be an integer.  Suppose there exists $x_1,\ldots,x_k\in\Q_p$ such that 
$$a=\bN(x_1,\ldots,x_k).$$
Then there exist $x_v\in K_v^*$, for $v|p$, such that 
$$\prod_{v|p}N_{K_v/\Q_p}(x_v)=a.$$
\end{lem}  
\begin{proof}
Since $a\ne 0$ we know that 
$$(x_1,\ldots,x_k)\ne 0.$$

Let $x^{(n)}$ be a sequence in $\Q^k$ which converges $p$-adically to $(x_1,\ldots,x_k)$.  Let $\omega_1,\ldots,\omega_k$ be the basis of $K/\Q$ used to construct the norm form $\bN$ and define $y^{(n)}\in K$ by 
$$y^{(n)}=x^{(n)}_1\omega_1+\ldots +x^{(n)}_k \omega_k.$$
For each $v|p$ write $y^{(n)}_v$ for the image of $y^{(n)}$ under the embedding of $K$ into $K_v$.
The sequence $y^{(n)}_v$ converges to some $x_v\in K_v^*$.  

We now have 
$$\prod_{v|p}N_{K_v/\Q_p}(x_v)=\lim_{n\rightarrow \infty}\prod_{v|p}N_{K_v/\Q_p}(y^{(n)}_v).$$
However, since $y^{(n)}\in K$ it follows, (for example by Gras \cite[Proposition 2.2, page 93]{gras}), that 
$$\prod_{v|p}N_{K_v/\Q_p}(y^{(n)}_v)=N_{K/\Q}(y^{(n)})=\bN(x^{(n)}_1,\ldots, x^{(n)}_k).$$
We conclude, by continuity of $\bN$, that 
$$\prod_{v|p}N_{K_v/\Q_p}(x_v)=\bN(x_1,\ldots,x_k)=a.$$
\end{proof}

\begin{lem}\label{badprimes}
Let $p$ be a prime for which (\ref{eq:diophantine}) has a solution in $\Q_p$.  There exist $a_0,b_0\in\Z$ and $l\in\N$, all depending on $p$, satisfying 
$$b_0,f(a_0,b_0)\ne 0\pmod{p^l},$$
such that for any $a,b\in\Z$ with 
$$a\equiv a_0\pmod{p^l}\text{ and }b\equiv b_0\pmod{p^l}$$
we have 
$$b,f(a,b)\in \bN(\Q_p^k)\setminus\{0\}.$$   
\end{lem}

\begin{proof}
For the duration of this proof let 
$$N=\bN(\Q_p^k)\setminus \{0\}.$$

By assumption there exist $a_1,b_1\in\Z_p$ with $b_1\ne 0$ such that 
$$f(\frac{a_1}{b_1})\in N.$$
By replacing $(a_1,b_1)$ by $(b_1^{k-1}a_1,b_1^k)$ we may assume that 
$b_1\in N$
and therefore
$f(a_1,b_1)\in N$.
The set $N$ is open and $f$ is continuous with respect to the $p$-adic topology.  It follows that there exists $\delta>0$ such that for any $a,b\in \Z_p$ with 
\begin{equation}\label{hyp}
|a-a_1|,|b-b_1|<\delta
\end{equation}
we have 
$$b,f(a,b)\in N.$$
For an $l\in\N$ which is sufficiently large in terms of $\delta$ the hypotheses (\ref{hyp}) are equivalent to 
$$a\equiv a_1\pmod {p^l},\;\; b\equiv b_1\pmod {p^l}.$$
The result follows on taking $a_0,b_0\in\Z$ which are congruent modulo $p^l$ to $a_1,b_1$.  Since 
$$b_1,f(a_1,b_1)\ne 0$$
we can guarantee that 
$$b_0,f(a_0,b_0)\ne 0\pmod{p^l}$$
provided $l$ is large enough.  
\end{proof}

We may now use all the previous lemmas to reduce our original problem to one involving prime divisors of $b$ and $f(a,b)$.

\begin{lem}
Suppose that (\ref{eq:diophantine}) has solutions in every $\Q_p$ and in $\R$.  Let $\p_1$ be a finite set of primes which contains the ramified primes in $K/\Q$ as well as the finitely many primes $p\not\equiv 1\pmod q$ for which the degrees $f_i$ are not coprime.  Then there exists a $\Delta\in\N$, divisible only by primes in $\p_1$, and integers $a_0,b_0$ such that if $p\in\p_1$ and $p^l$ is the maximal power of $p$ dividing $\Delta$ then 
$$b_0,f(a_0,b_0)\not\equiv 0\pmod{p^l}$$
and the following implication is true.

Suppose that $a,b$ are integers for which the following hold:

\begin{enumerate}
\item We have 
\begin{equation}\label{congruence}
a\equiv a_0\pmod \Delta\text{ and }b\equiv b_0\pmod \Delta.
\end{equation}

\item Each prime $p$ with $p|bf(a,b)$ and $p\notin \p_1$ satisfies 
$$p\not\equiv 1\pmod q.$$

\item We have $b\geq 0$ and $f(a,b)\geq 0$.
\end{enumerate} 
Then (\ref{eq:diophantine}) has a solution over $\Q$.
\end{lem} 
\begin{proof}
By Lemma \ref{homogeneous} it is sufficient to show that 
$$b,f(a,b)\in N_{K/\Q}(K^*).$$
By the Hasse norm principle this is equivalent to showing that 
$$b,f(a,b)\in N_{K/\Q}(I_K).$$
We must therefore show, for all places of $\Q$, that $b$ and $f(a,b)$ are products of local norms.

\begin{enumerate}
\item For each prime $p\in \p_1$ we may use Lemma \ref{badprimes} to construct $l_p,a_{0,p},a_{1,p}$. These will satisfy
$$b_{0,p},f(a_{0,p},b_{0,p})\not\equiv0\pmod{p^{l_p}}.$$
We now let 
$$\Delta=\prod_{p\in\p_1}p^{l_p}$$
and use the Chinese remainder theorem to construct $a_0,b_0$ satisfying 
$$a_0\equiv a_{0,p}\pmod {p^{l_p}},\;\; b_0\equiv b_{0,p}\pmod {p^{l_p}}$$
for all $p\in\p_1$.  It follows by our assumption (\ref{congruence}) and Lemma \ref{badprimes} that 
$$b,f(a,b)\in \bN(\Q_p^k)\setminus\{0\}.$$
We conclude, using Lemma \ref{equivalentlocal}, that $b$ and $f(a,b)$ are suitable products of local norms for all primes in $\p_1$.

\item For primes not in $\p_1$ we know that either $p\nmid b$, in which case we use Lemma \ref{unramified} to write $b$ as a suitable product of local norms, or $p|b$. In the latter situation $p\not\equiv 1\pmod q$ and therefore the degrees of the prime ideals above $p$ are coprime.  The required local condition for $b$ now follows by Lemma \ref{normp}.  We may use an identical argument for $f(a,b)$.

\item Finally we consider the infinite place.  Since $b,f(a,b)\geq 0$ they are both local norms at infinity.
\end{enumerate}
The above cases cover all the places of $\Q$ so the result follows.
\end{proof}

For the remainder of this paper we let $\p_1$ be a finite set of primes including those which are ramified in $K/\Q$ or which divide the coefficients of $a^3$ or $b^3$ in the polynomial $f(a,b)$ or which divide the discriminant of $f$.  In addition $\p_1$ will contain those primes $p\not\equiv 1\pmod q$ for which the degrees $f_i$ are not coprime.  We also include in $\p_1$ the prime $q$ and all  primes up to some absolute constant $P_1$, (which will be determined in Lemma \ref{notsquarefree} below).  We let $a_0,b_0,\Delta$ be the quantities constructed in the last lemma and use the notation $C(a,b)$ to denote that $a,b$ satisfy (\ref{congruence}).  

Since $f$ is a cubic, we can, without loss of generality, apply a linear change of variable over $\Q$ to guarantee that its leading coefficient is positive and all its real roots are negative.  We may thus assume that if $x>0$ then $f(x)>0$.  In particular, if $a,b>0$ then $f(a,b)>0$.  
For a large $N$ we will apply a sieve to count pairs $(a,b)\in (0,N]^2$ satisfying $C(a,b)$ for which $bf(a,b)$ has no prime factor $p\notin\p_1$ with $p\equiv 1\pmod q$.  If we can prove a positive lower bound for this quantity then it follows by the last lemma that (\ref{eq:diophantine}) has a solution over $\Q$.

\section{Levels of Distribution}

We need various level of distribution results for the values $bf(a,b)$. All implied constants in this section may depend on the polynomial $f$ and on $\Delta$. 

The main result of this section, Lemma \ref{mainlod}, is proved using very similar methods to those of  Daniel, \cite[Lemmas 3.2 and 3.3]{daniel}. We extend his results to handle the form $bf(a,b)$, rather than $f(a,b)$,  with $a,b$ in a fixed arithmetic progression.  Let $\cR$ be a compact region of $\R^2$.  We begin by considering the quantity 
$$R^*(\cR,d_1,d_2)=\#\{(a,b)\in \cR:C(a,b),(a;b;d)=1,d_1|f(a,b),d_2|bf(a,b)\}.$$
Throughout this paper we are using $(a;b)$ to denote the greatest common factor.  We need only consider $R^*(\cR,d_1,d_2)$ for $d_1,d_2\in\N$ satisfying 
$$(d_1;d_2)=(d_1d_2;\Delta)=1.$$
We will write $d=d_1d_2$.

We say that points $(a_1,b_1),(a_2,b_2)$ with 
$$(a_1;b_1;d)=(a_2;b_2;d)=1$$
are equivalent modulo $d$ if 
$$(a_1,b_1)\equiv \lambda(a_2,b_2)\pmod d$$
for some $\lambda\in\Z$ which must necessarily satisfy $(\lambda;d)=1$.  The restriction to points with $(a;b;d)=1$ is required to make our notion of equivalence into a valid equivalence relation, (since $(\lambda;d)=1$ it has an inverse mod $d$).  We will call points with $(a;b;d)=1$ primitive modulo $d$.  The number of primitive points in each equivalence class which are distinct modulo $d$ is $\varphi(d)$.

Observe that the properties $f(a,b)\equiv 0\pmod{d_1}$ and $bf(a,b)\equiv0\pmod{d_2}$ are preserved by equivalence.  We may therefore define $\mathcal U(d_1,d_2)$ to be the set of equivalence classes of primitive points modulo $d=d_1d_2$ for which 
$f(a,b)\equiv 0\pmod{d_1}$
and 
$bf(a,b)\equiv 0\pmod{d_2}$.

For an equivalence class $x\pmod d$ we let $\lambda(x)$ be the lattice in $\Z^2$ generated by the points of $x$.  Thus $y\in \lambda(x)$ if and only if there exists some $(a,b)\in x$ and $\lambda\in \Z$ with 
$$y\equiv \lambda(a,b)\pmod d.$$
In particular the primitive points in $\lambda(x)$ are precisely those in $x$.  It follows that 
\begin{eqnarray*}
R^*(\cR,d_1,d_2)&=&\sum_{x\in\mathcal U(d_1,d_2)}\#\{(a,b)\in\cR\cap x:C(a,b)\}\\
&=&\sum_{x\in\mathcal U(d_1,d_2)}\#\{(a,b)\in \cR\cap\lambda(x):C(a,b),(a;b;d)=1\}\\
&=&\sum_{x\in\mathcal U(d_1,d_2)}\twosum{(a,b)\in \cR\cap\lambda(x)}{C(a,b)}\sum_{e|(a;b;d)}\mu(e)\\
&=&\sum_{x\in\mathcal U(d_1,d_2)}\sum_{e|d}\mu(e)\#\{(a,b)\in \cR\cap \lambda(x):C(a,b),e|(a,b)\}\\
&=&\sum_{x\in\mathcal U(d_1,d_2)}\sum_{e|d}\mu(e)\#\{(a,b)\in\cR\cap \lambda(x,e):C(a,b)\}\\
\end{eqnarray*}
where $\lambda(x,e)$ is the sublattice of $\lambda(x)$ consisting of points divisible by $e$.  

We have $(d;\Delta)=1$ so $(e;\Delta)=1$.  It follows that the sublattice of $\lambda(x,e)$ consisting of those points which are divisible by $\Delta$ is precisely $\lambda(x,e\Delta)$. It is then clear that the set 
$$\{(a,b)\in \lambda(x,e):C(a,b)\}$$ 
is a coset of the lattice $\lambda(x,e\Delta)$.

\begin{lem}
We have 
$$|\det \lambda(x,e\Delta)|=de\Delta^2.$$
\end{lem}
\begin{proof}
In general, if a lattice in $\Z^2$ is formed from all points whose reduction mod $n$ is in a set of $c$ equivalence classes then its determinant is $\frac{n^2}{c}$.

For our specific problem we take $n=d\Delta$.  Let $(a,b)$ be a fixed point of $x$. Since $(a;b;d)=1$ we know that the number of points modulo $d$ which are multiples of $(a,b)$ and divisible by $e$ is $\frac{d}{e}$.  Since $(d;\Delta)=1$ it follows by the Chinese remainder theorem that $c=\frac{d}{e}$ and therefore 
$$|\det\lambda(x,e\Delta)|=\frac{d^2\Delta^2}{d/e}=de\Delta^2.$$
\end{proof}

Let $R_1(x,e\Delta)$ denote the length of the shortest nonzero vector in $\lambda(x,e\Delta)$.  It is clear that this is bounded below by $R_1(x)$, the length of the shortest nonzero vector in $\lambda(x)$.  Let $V(\cR)$ and $P(\cR)$ denote the volume and perimeter of $\cR$, respectively.  By the standard method for counting lattice points we get 
$$R^*(\cR,d_1,d_2)=\sum_{x\in\mathcal U(d_1,d_2)}\sum_{e|d}\mu(e)
\left(\frac{V(\cR)}{de\Delta^2}+O(1+\frac{P(\cR)}{R_1(x)})\right).$$
Let $\rho^*(d_1,d_2)$ denote the number of primitive solutions modulo $d$ to 
$f(a,b)\equiv 0\pmod{d_1}$
and
$bf(a,b)\equiv 0\pmod{d_2}$.
Since the number of distinct points modulo $d$ in each equivalence class is $\varphi(d)$ we have 
$$\sum_{x\in\mathcal U(d_1,d_2)}\sum_{e|d}\frac{\mu(e)}{e}=\sum_{x\in\mathcal U(d_1,d_2)}\frac{\varphi(d)}{d}=\frac{\rho^*(d_1,d_2)}{d}.$$
We conclude that for any $\epsilon>0$ we have 
$$R^*(\cR,d_1,d_2)=\frac{\rho^*(d_1,d_2)V(\cR)}{d^2\Delta^2}+O_\epsilon(d^\epsilon(1+P(\cR)\sum_{x\in\mathcal U(d_1,d_2)}R_1(x)^{-1})).$$
Averaging this over $d_1$ and $d_2$ we get 
\begin{eqnarray*}
\lefteqn{\twosum{d_1\leq D_1,d_2\leq D_2}{(d_1;d_2)=(d_1d_2;\Delta)=1}\max_{P(\cR)\leq M}|R^*(\cR,d_1,d_2)-\frac{\rho^*(d_1,d_2)V(\cR)}{d_1^2d_2^2\Delta^2}|}\hspace{2cm}\\
&\ll_\epsilon&(D_1D_2)^\epsilon(D_1D_2+M\twosum{d_1\leq D_1,d_2\leq D_2}{(d_1;d_2)=(d_1d_2;\Delta)=1}\sum_{x\in\mathcal U(d_1,d_2)}R_1(x)^{-1}).\\
\end{eqnarray*}
Let $v_1(x)$ denote the shortest nonzero vector in $\lambda(x)$.  We know that 
$$\|v_1(x)\|^2\ll |\det \lambda(x)|\ll D_1D_2.$$
We may thus write the final sum as 
$$\sum_{0<a^2+b^2\ll D_1D_2}\frac{1}{\sqrt{a^2+b^2}}\#\{d_1,d_2,x:(d_1;d_2)=(d_1d_2;\Delta)=1,v_1(x)=(a,b)\}.$$
If $v_1(x)=(a,b)$ then 
$$d_1d_2|bf(a,b).$$
We first consider the contribution to the above sum from pairs $(a,b)$ with $b\ne 0$. Since $f$ is irreducible we have $bf(a,b)\ne 0$. It follows that the number of $d_1,d_2$ for a given $(a,b)$ is bounded by 
$$\tau_3(bf(a,b))\ll_\epsilon (D_1D_2)^\epsilon.$$  
For each such $d_1,d_2$ the number of $x\in \mathcal U(d_1,d_2)$ for which $v_1(x)=(a,b)$ is at most 
$$\#\mathcal U(d_1,d_2)=\frac{\rho^*(d_1,d_2)}{\varphi(d_1d_2)}\ll_\epsilon (D_1D_2)^\epsilon.$$
We conclude that 
\begin{eqnarray*}
\lefteqn{\twosum{0<a^2+b^2\ll D_1D_2}{b\ne 0}\frac{1}{\sqrt{a^2+b^2}}\,\#\{d_1,d_2,x:(d_1;d_2)=(d_1d_2;\Delta)=1,v_1(x)=(a,b)\}}\\
&\ll_\epsilon& (D_1D_2)^\epsilon\sum_{0<a^2+b^2\ll D_1D_2}\frac{1}{\sqrt{a^2+b^2}}\hspace{8cm}\\
&\ll_\epsilon&(D_1D_2)^{\frac12+\epsilon}.
\end{eqnarray*}
It remains to estimate the contribution from pairs $(a,0)$:
$$\sum_{0<a\ll \sqrt{D_1D_2}}\frac{1}{a}\,\#\{d_1,d_2,x:(d_1;d_2)=(d_1d_2;\Delta)=1,v_1(x)=(a,0)\}.$$
Suppose that $v_1(x)=(a,0)$.  We then have 
$$f(a,0)\equiv 0\pmod{d_1}.$$
Since $f(a,0)\ne 0$ it follows that the number of possible $d_1$ is bounded by 
$$\tau(f(a,0))\ll D_1^\epsilon.$$
For each such $d_1$ the number of $d_2$ is clearly bounded by $D_2$.  As above, the number of $x$ is then $O_\epsilon((D_1D_2)^\epsilon)$.  We conclude that 
\begin{eqnarray*}
\lefteqn{\sum_{0<a\ll \sqrt{D_1D_2}}\frac{1}{a}\,\#\{d_1,d_2,x:(d_1;d_2)=(d_1d_2;\Delta)=1,v_1(x)=(a,0)\}}\hspace{4cm}\\
&\ll_\epsilon &D_1^\epsilon D_2^{1+\epsilon}\sum_{0<a\ll \sqrt{D_1D_2}}\frac{1}{a}\ll_\epsilon D_1^\epsilon D_2^{1+\epsilon}.
\end{eqnarray*}
Combining these two cases we get 
$$\sum_{0<a^2+b^2\ll D_1D_2}\frac{1}{\sqrt{a^2+b^2}}\,\#\{d_1,d_2,x:(d_1;d_2)=(d_1d_2;\Delta)=1,v_1(x)=(a,b)\}$$
$$\ll_\epsilon (D_1D_2)^\epsilon((D_1D_2)^{\frac12}+D_2).$$
We have therefore proved the following.

\begin{lem}\label{Rstar}
For any $D_1,D_2>0$ and any $\epsilon>0$ we have 
$$\twosum{d_1\leq D_1,d_2\leq D_2}{(d_1;d_2)=(d_1d_2;\Delta)=1}\max_{P(\cR)\leq M}|R^*(\cR,d_1,d_2)-\frac{\rho^*(d_1,d_2)V(\cR)}{d_1^2d_2^2\Delta^2}|$$
$$\ll_\epsilon (D_1D_2)^\epsilon(D_1D_2+M((D_1D_2)^{\frac12}+D_2)).$$
\end{lem}

We next remove the restriction to primitive points.  As in Daniel's work, \cite[Lemma 3.3]{daniel}, we need the multiplicative functions $\psi_k$ which map the prime power $p^{\alpha k+\beta}$, for $1\leq \beta\leq k$, to $p^{\alpha+1}$.  

Let
\begin{eqnarray*}
R(\cR,d_1,d_2)&=&\#\{(a,b)\in\cR:C(a,b),d_1|f(a,b),d_2|bf(a,b)\}\\
&=&\twosum{e_1|\psi_3(d_1)}{e_2|\psi_4(d_2)}N(d_1,d_2,e_1,e_2)
\end{eqnarray*}
where
\begin{eqnarray*}
\lefteqn{N(d_1,d_2,e_1,e_2)}\\
&=&\#\{(a,b)\in\cR:C(a,b),(a;b;\psi_3(d_1)\psi_4(d_2))=e_1e_2,
d_1|f(a,b),d_2|bf(a,b)\}\\ 
&=&\#\{(a,b)\in\cR/e_1e_2:C(e_1e_2(a,b)),(a;b;\frac{\psi_3(d_1)\psi_4(d_2)}{e_1e_2})=1,\\
&&\hspace{3cm}\frac{d_1}{(d_1;e_1^3)}|f(a,b),\frac{d_2}{(d_2;e_2^4)}|bf(a,b)\}\\ 
&=&\#\{(a,b)\in\cR/e_1e_2:C(e_1e_2(a,b)),(a;b;\frac{d_1d_2}{(d_1;e_1^3)(d_2;e_2^4)})=1,\\
&&\hspace{3cm}\frac{d_1}{(d_1;e_1^3)}|f(a,b),\frac{d_2}{(d_2;e_2^4)}|bf(a,b)\}\\ 
&=&R^*(\cR/e_1e_,\frac{d_1}{(d_1;e_1^3)},\frac{d_2}{(d_2;e_2^4)};e_1,e_2).
\end{eqnarray*}
Thus
$$R(\cR,d_1,d_2)=\twosum{e_1|\psi_3(d_1)}{e_2|\psi_4(d_2)}R^*(\cR/e_1e_,\frac{d_1}{(d_1;e_1^3)},\frac{d_2}{(d_2;e_2^4)};e_1,e_2).$$

Here the addition of $(;e_1,e_2)$ to $R^*$ denotes that the congruences $C(a,b)$ are replaced by $C(e_1e_2(a,b))$.  Since $(e_1e_2;\Delta)=1$ these congruences are
$$a\equiv \overline{e_1e_2}a_0\pmod \Delta$$
and 
$$b\equiv \overline{e_1e_2}b_0\pmod \Delta.$$
The precise choice of coset has no effect on the above analysis of $R^*$ so Lemma \ref{Rstar} still holds when different congruence classes are taken for each pair $d_1,d_2$ in the sum.

Let $\rho(d_1,d_2)$ be the number of solutions modulo $d_1d_2$ to 
$$f(a,b)\equiv 0\pmod{d_1},\quad bf(a,b)\equiv 0\pmod{d_2}.$$
Applying the above analysis to the region $(0,d_1d_2]^2$ with no congruence $C$ gives the decomposition 
$$\rho(d_1,d_2)=\twosum{e_1|\psi_3(d_1)}{e_2|\psi_4(d_2)}(\frac{(d_1;e_1^3)}{e_1}\frac{(d_2;e_2^4)}{e_2})^2\rho^*(\frac{d_1}{(d_1;e_1^3)},\frac{d_2}{(d_2;e_2^4)}).$$
It follows that 
\begin{eqnarray*}
\lefteqn{R(\cR,d_1,d_2)-\frac{\rho(d_1,d_2)V(\cR)}{d_1^2d_2^2\Delta^2}}\\
&=&\twosum{e_1|\psi_3(d_1)}{e_2|\psi_4(d_2)}(R^*(\cR/e_1e_2,\frac{d_1}{(d_1;e_1^3)},\frac{d_2}{(d_2;e_2^4)})-\frac{V(\cR/e_1e_2)(d_1;e_1^3)^2(d_2;e_2^4)^2\rho^*(\frac{d_1}{(d_1;e_1^3)},\frac{d_2}{(d_2;e_2^4)})}{d_1^2d_2^2\Delta^2}.\\
\end{eqnarray*}
We are interested in the average of this over $d_1\leq D_1,d_2\leq D_2$ so we consider 
$$\twosum{e_1f_1\leq D_1,e_2f_2\leq D_2}{(e_1f_1;e_2f_2)=(e_1f_1e_2f_2;\Delta)=1}\delta(e_1,e_2,f_1,f_2)\max_{P(\cR)\leq M}|R^*(\cR/e_1e_2,f_1,f_2)-\frac{\rho^*(f_1,f_2)V(\cR/e_1e_2)}{f_1^2f_2^2}|,$$
where $\delta(e_1,e_2,f_1,f_2)$ is the number of pairs $d_1\leq D_1,d_2\leq D_2$ with 
$$e_1|\psi(d_1),\;e_2|\psi(d_2),\;f_1=\frac{d_1}{(d_1;e_1^3)},\; f_2=\frac{d_2}{(d_2;e_2^4)}.$$
It is clear that $\delta$ is the product of the number of suitable $d_1$ by the number of $d_2$.  These latter quantities were estimated by Daniel: they are bounded by divisor functions.  It follows that for any $\epsilon>0$ we have 
$$\delta(e_1,e_2,f_1,f_2)\ll_\epsilon (D_1D_2)^\epsilon.$$
Our sum is thus majorised by 
$$(D_1D_2)^\epsilon\twosum{e_1f_1\leq D_1,e_2f_2\leq D_2}{(e_1f_1;e_2f_2)=(e_1f_1e_2f_2;\Delta)=1}\max_{P(\cR)\leq M}|R^*(\cR/e_1e_2,f_1,f_2)-\frac{\rho^*(f_1,f_2)V(\cR/e_1e_2)}{f_1^2f_2^2}|.$$

For each pair $(e_1,e_2)$ in this sum we apply Lemma \ref{Rstar} to the sum over $f_1,f_2$.  This results in a bound 
$$(D_1D_2)^\epsilon\sum_{e_1\leq D_1,e_2\leq D_2}\left(\frac{D_1D_2}{e_1e_2}+\frac{M}{e_1e_2}\big((\frac{D_1D_2}{e_1e_2})^{\frac12}+\frac{D_2}{e_2}\big)\right).$$
We may therefore conclude with the following level of distribution result.

\begin{lem}
For any $D_1,D_2>0$ and any $\epsilon>0$ we have 
\begin{eqnarray*}
\lefteqn{\twosum{d_1\leq D_1,d_2\leq D_2}{(d_1;d_2)=(d_1d_2;\Delta)=1}\max_{P(\cR)\leq M}|R(\cR,d_1,d_2)-\frac{\rho(d_1,d_2)V(\cR)}{d_1^2d_2^2\Delta^2}|}\hspace{2cm}\\
&\ll_\epsilon &(D_1D_2)^\epsilon(D_1D_2+M((D_1D_2)^{\frac12}+D_2)).
\end{eqnarray*}
\end{lem}

We are interested in this result when $\cR=(0,N]^2$ for large $N$.

\begin{lem}\label{mainlod}
Let 
$$R(d_1,d_2)=R((0,N]^2,d_1,d_2).$$
  Suppose $\eta>0$ is fixed.  Then there exists $\delta>0$, depending on $\eta$, such that if 
$$0<D_1D_2\leq N^{2-\eta}$$
and 
$$0<D_2\leq N^{1-\eta}$$
we have 
$$\twosum{d_1\leq D_1,d_2\leq D_2}{(d_1;d_2)=(d_1d_2;\Delta)=1}|R(d_1,d_2)-\frac{\rho(d_1,d_2)N^2}{d_1^2d_2^2\Delta^2}|\ll_\eta N^{2-\delta}.$$
\end{lem}
\begin{proof}
This follows on putting $V(\cR)=N^2$, $P(\cR)\ll N$ into the previous lemma and taking $\epsilon$ sufficiently small in terms of $\eta$. 
\end{proof}

If we let $\rho_1(d)$ be the number of solutions modulo $d$ to 
$f(a,b)\equiv 0\pmod d$
and $\rho_2(d)$ the number of solutions to 
$bf(a,b)\equiv 0\pmod d$
then if $(d_1;d_2)=1$ we have 
$$\rho(d_1,d_2)=\rho_1(d_1)\rho_2(d_2).$$

We also need to understand the quantity 
$$R_1(d_1,d_2)=\#\{(a,b)\in (0,N]^2:C(a,b),\,bf(a,b)\equiv 0\!\!\!\pmod{d_1},\,b\equiv 0\!\!\!\pmod {d_2}\}.$$
This is only required for small $d_1,d_2$ so the following is sufficient.

\begin{lem}\label{fliplod}
For any $d_1,d_2\in\N$ with $(d_1;d_2)=(d_1d_2;\Delta)=1$ and $d_1d_2\leq N$ we have, for any $\epsilon>0$ that  
$$R_1(d_1,d_2)=\frac{N^2\rho_2(d_1)}{d_1^2d_2\Delta^2}+O_\epsilon(Nd_1^\epsilon).$$
\end{lem}
\begin{proof}
The number of points counted by $R_1$ congruent to a given solution modulo $d_1d_2\Delta$ is 
$$\frac{N^2}{d_1^2d_2^2\Delta^2}+O(1+\frac{N}{d_1d_2\Delta})=\frac{N^2}{d_1^2d_2^2\Delta^2}+O(\frac{N}{d_1d_2}).$$
By the Chinese remainder theorem the number of solutions modulo $d_1d_2\Delta$ is $d_2\rho_2(d_1)$. It follows that 
$$R(d_1,d_2)=\frac{N^2\rho_2(d_1)}{d_1^2d_2\Delta^2}+O(\frac{N\rho_2(d_1)}{d_1})=\frac{N^2\rho_2(d_1)}{d_1^2d_2\Delta^2}+O_\epsilon(Nd_1^\epsilon).$$
\end{proof}

\section{The Functions $\rho_1$ and $\rho_2$}

We need various estimates for sums and products involving the functions $\rho_1$ and $\rho_2$.  Let $\nu(d)$ be the number of solutions to the congruence 
$$f(x)\equiv 0\pmod d.$$
For all primes $p\notin\p_1$ we may write $\rho_1(p)$ and $\rho_2(p)$ in terms of $\nu(p)$:
$$\rho_1(p)=(p-1)\nu(p)+1$$
and 
$$\rho_2(p)=(p-1)\nu(p)+p.$$

In the following equations let $c$ denote a real constant which may depend  on $q$ and $f$ and which may differ from line to line.  It is well known that 
$$\sum_{p\leq x}\frac{1}{p}=\log\log x+c+o(1),$$
$$\twosum{p\leq x}{p\equiv a\pmod q}\frac{1}{p}=\frac{1}{q-1}\log\log x+c+o(1),$$
if $(a;q)=1$, and 
$$\sum_{p\leq x}\frac{1}{p^2}=c+o(1).$$

Let $L$ be the cubic field generated by the polynomial $f$ and let $\zeta_L$ be its Dedekind zeta function.  For all primes $p\notin\p_1$ we know that $\nu(p)$ is equal to the coefficient of $p^{-s}$ in $\zeta_L(s)$.  It follows from the Prime Ideal Theorem that 
$$\sum_{p\leq x}\frac{\nu(p)}{p}=\log\log x+c+o(1).$$
Finally we would like to show that, for $(a;q)=1$, we have 
\begin{equation}\label{nuprog}
\twosum{p\leq x}{p\equiv a\pmod q}\frac{\nu(p)}{p}\sim \frac{1}{q-1}\log\log x.
\end{equation}
Unfortunately this is not always true.  For example, suppose we have 
$$f(t)=t^3-7t^2+14t-7.$$
The field $L$ is then abelian of degree $3$ and contained in the cyclotomic field $\Q(\zeta_7)$.  It is easy to deduce from this that 
$$\nu(p)=\begin{cases}
3 & p\equiv \pm 1\pmod 7\\
0 & p\equiv \pm 2,\pm 3\pmod 7.\\
\end{cases}$$
The formula (\ref{nuprog}) is therefore not true for this $f$ when $q=7$.  It follows that many of the details of the sieve would be different in this case.  In order to avoid these difficulties we restrict our attention to those polynomials $f$ and primes $q$ for which (\ref{nuprog}) holds.  We will show that (\ref{nuprog}) follows from our hypothesis that the number field $L$ is not contained in $\Q(\zeta_q)$.

Expanding using characters we are interested in 
$$\frac{1}{q-1}\sum_{\chi\pmod q}\overline{\chi(a)}\sum_{p\leq x}\frac{\chi(p)\nu(p)}{p}.$$
For $p\notin\p_1$ the quantity $\chi(p)\nu(p)$ is the coefficient of $p^{-s}$ in the function $\zeta_L(s,\chi)$. This is the  Hecke $L$-function coming from the character which maps an ideal $I$ to $\chi(N(I))$.  

\begin{lem}\label{hecke}
If $\chi\ne \chi_0$ is a character modulo $q$ then, under our assumption that $L\not\subseteq \Q(\zeta_q)$, $\zeta_L(s,\chi)$ is regular  at $s=1$.
\end{lem}
\begin{proof}
We say that a property holds for almost all primes if it holds for all primes with finitely many exceptions. It is  enough to show that the Hecke character $I\mapsto \chi(N(I))$ is not induced from the trivial character, as it is the only primitive Hecke character whose $L$-function has a singularity at $s=1$.  In other words we need to show that there are infinitely many prime ideals $P$ for which $\chi(N(P))\ne 1$.  We suppose that this is false so that, in particular,  almost all primes $p$,  for which there is an ideal of norm $p$,  are in a proper subgroup $H$ of $(\Z/q\Z)^*$.  

We first consider the case that $L/\Q$ is not Galois,   so its discriminant, $\delta$, is not a square.  It can be shown that if a prime $p$ satisfies 
$(\frac{\delta}{p})=-1$
then it factorises in $L$ into prime ideals of norms $p$ and $p^2$.  It follows that the reduction modulo $q$ of almost all such primes must be in $H$.  However, we can show using Dirichlet's theorem on primes in arithmetic progressions that for any prime $q$ and any nonsquare integer $\delta$ the reductions modulo $q$ of almost all the primes $p$ for which $(\frac{\delta}{p})=-1$ generate the whole of $(\Z/q\Z)^*$.  This is a contradiction so $\chi$ cannot be induced from the trivial character and $\zeta_L(s,\chi)$ is regular at $1$.

Next we consider the case that $L/\Q$ is Galois.  We know, by assumption, that the primes which split completely in $L$ are contained in $H$.  

Suppose in general that  we have Galois number fields $L_1,L_2$ and almost all the primes which split completely in $L_1$ also split completely in $L_2$.  It follows that almost all the primes which split in $L_1$ also split in the composite extension $L_1L_2$.  By Chebotarev's Density Theorem the density of primes which split in $L_1$ is $\frac{1}{[L_1:\Q]}$ whereas the density of those splitting in $L_1L_2$ is $\frac{1}{[L_1L_2:\Q]}$.  We conclude that 
$$\frac{1}{[L_1:\Q]}\leq \frac{1}{[L_1L_2:\Q]}$$
so that $L_1L_2=L_1$ and therefore $L_2$ is a subfield of $L_1$.

By class field theory we can construct a number field $L_H$ whose only ramified prime is $q$ and for which the primes which split completely are those in $H$, ($L_H$ is the class field coming from the modulus $(q)$ and the subgroup $H$).  It follows by the previous paragraph that $L_H$ is contained in $L$.  However, $[L:\Q]=3$ so we must have $L=L_h$.  By class field theory, $L_H\subseteq \Q(\zeta_q)$, which contradicts our assumption on $L$.  We deduce that the Hecke character is not induced from the trivial one and thus its $L$-function has no singularities.  
\end{proof}

It now follows by general theory that for $\chi\ne\chi_0$ we have, as $x\rightarrow \infty$,  
$$\sum_{p\leq x}\frac{\chi(p)\nu(p)}{p}=c+o(1),$$
for some constant $c$ depending on $f,q$ and $\chi$.  We can therefore conclude that 
$$\twosum{p\leq x}{p\equiv a\pmod q}\frac{\nu(p)}{p}=\frac{1}{q-1}\log\log x+c+o(1),$$
with $c$ depending on $f,q$ and $a$.

Let
$$\p=\{p:p\notin \p_1,p\equiv 1\pmod q\}$$
and
$$\p'=\{p:p\notin \p_1,p\not\equiv 1\pmod q\}.$$

\begin{lem}\label{rhoprod}
As $x\rightarrow \infty$ we have 
$$\twoprod{p\leq x}{p\in\p}(1-\frac{\rho_2(p)}{p^2})\sim \frac{c_2(f,q)}{(\log x)^{\frac{2}{q-1}}},$$
and
$$\twoprod{p\leq x}{p\in\p'}(1-\frac{\rho_1(p)}{p^2})\sim \frac{c_1(f,q)}{(\log x)^{\frac{q-2}{q-1}}}$$
where $c_1(f,q),c_2(f,q)>0$.
\end{lem}
\begin{proof}
Let 
$$P=\twoprod{p\leq x}{p\in\p}(1-\frac{\rho_2(p)}{p^2}).$$
Since all the primes dividing $f$ are in $\p_1$ and hence not in $\p$ we know that $\rho_2(p)<p^2$ for all $p\in\p$.  It follows that the terms in $P$ are all positive so we may take logs:
\begin{eqnarray*}
\log P&=&\twosum{p\leq x}{p\in\p}\log(1-\frac{\rho_2(p)}{p^2})\\
&=&\twosum{p\leq x}{p\in\p}(-\frac{\rho_2(p)}{p^2}+O(\frac{1}{p^2}))\\
&=&\twosum{p\leq x}{p\in\p}(-\frac{\nu(p)+1}{p}+O(\frac{1}{p^2}))\\
&=&-\frac{2}{q-1}\log\log x+c+o(1).\\
\end{eqnarray*}
The first result follows on taking $c_2(f,q)=e^c$ with the $c$ from the last line.  The second result can be proved analogously.
\end{proof}

It is clear that if $p\notin\p_1$ then $\nu(p)\leq 3$ so that $\rho_1(p)\leq 3p$.  We also need a bound for $\rho_1$ at prime powers.

\begin{lem}\label{rho1bound}
For any prime $p\notin\p_1$ and any $\alpha\in\N$ we have 
$$\rho_1(p^\alpha)\ll p^{\frac{4\alpha}{3}},$$
the implied constant being absolute.
\end{lem}
\begin{proof}
We substitute Daniel's bound \cite[(3.2)]{daniel}, which holds for all $p\notin\p_1$, into his identity \cite[(7.4)]{daniel}.  This results in 
\begin{eqnarray*}
\rho_1(p^\alpha)&\leq& 3\sum_{0\leq \beta<\lceil\alpha/3\rceil}p^{\alpha+\beta}+p^{2(\alpha-\lceil\alpha/3\rceil)}\\
&\leq&3\frac{p^{\alpha+\lceil\alpha/3\rceil}}{p-1}+p^{2(\alpha-\lceil\alpha/3\rceil)}\ll p^{\frac{4\alpha}{3}}.\\
\end{eqnarray*}
\end{proof}

As a consequence of this we see that for any $r$ with no prime factors in $\p_1$ we have 
$$\rho_1(r) \ll r^{\frac43}.$$

\section{The Sum of a Multiplicative Function in an Arithmetic Progression}

Let $g$ be a nonnegative multiplicative function supported on squarefree numbers which satisfies  
\begin{equation}\label{mult1}
\sum_{p\leq x}g(p)\log p=k\log x+O(1),
\end{equation}
for some $k>0$.  If $2\leq w<z$ we assume that 
\begin{equation}\label{mult2}
\prod_{w\leq p<z}(1+g(p))\ll \left(\frac{\log z}{\log w}\right)^k.
\end{equation}
We also suppose that 
\begin{equation}\label{mult3}
\sum_p g(p)^2 \log p<\infty.
\end{equation}
Under these assumptions Friedlander and Iwaniec, \cite[Theorem A.5]{opera}, show that 
\begin{equation}
  \label{mult4}
\sum_{m\leq x}g(m)=c_g(\log x)^k+O((\log x)^{k-1}),
\end{equation}
where
$$c_g=\frac{1}{\Gamma(k+1)}\prod_p(1-\frac{1}{p})^k(1+g(p)).$$
We require the following modified version of this result.

\begin{lem}\label{multlem}
Let $g$ be a nonnegative multiplicative function supported on squarefree numbers which satisfies (\ref{mult1}), (\ref{mult2}) and (\ref{mult3}) for some $k>0$.  Let $q>2$ be prime.  Suppose that $g(q)=0$ and that for all primes $p\equiv1\pmod q$ we have $g(p)=0$.  Finally suppose that if $a\not\equiv 0,1\pmod q$ then 
\begin{equation}\label{mult5}
\twosum{p\leq x}{p\equiv a\pmod q}g(p)\log p=\frac{1}{q-2}\,k\log x+O(1).
\end{equation}
Then, for any $a$ with $(a;q)=1$ we have 
$$\twosum{m\leq x}{m\equiv a\pmod q}g(m)=(\frac{c_g}{q-1}+o(1))(\log x)^k.$$
\end{lem}
\begin{proof}
Let 
$$M_g(x)=\twosum{m\leq x}{m\equiv a\pmod q}g(m).$$
We begin by considering 
$$\twosum{m\leq x}{m\equiv a\pmod q}g(m) \log m=\twosum{np\leq x}{np\equiv a\pmod q}g(np)\log p.$$
Using that $g$ is multiplicative and supported on squarefree numbers coprime to $q$ this can be written as 
$$\sum_{n\leq x}g(n)\twosum{p\leq x/n}{p\equiv a\overline n\pmod q}g(p)\log p-\twosum{np^2\leq x}{np^2\equiv a\pmod q}g(np)g(p)\log p.$$
From (\ref{mult2}) we get 
$$\sum_{n\leq x}g(n)\leq \prod_{p\leq x}(1+g(p))\ll (\log x)^k$$
and from (\ref{mult3}) we deduce  
$$\sum_{np^2\leq x}g(np)g(p)\log p\ll \sum_{n\leq x}g(n)\sum_p g(p)^2\log p\ll (\log x)^k.$$
Using these bounds as well  as (\ref{mult5}) the above sum becomes 
$$\frac{k}{q-2}\twosum{n\leq x}{n\not\equiv a\pmod q}g(n)(\log x-\log n)+O((\log x)^k).$$
We can write this as
$$\frac{k}{q-2}\left(\sum_{n\leq x}g(n)(\log x-\log n)-\twosum{n\leq x}{n\equiv a\pmod q}g(n)(\log x-\log n)\right)+O((\log x)^k).$$
From the bound (\ref{mult4}) we have 
$$\sum_{n\leq x}g(n)=c_g(\log x)^k+O((\log x)^{k-1})$$
and, summing by parts, 
$$\sum_{n\leq x}g(n)\log n=\frac{kc_g}{k+1}(\log x)^{k+1}+O((\log x)^k)).$$
Our sum is thus 
$$\frac{k}{q-2}\left(\frac{c_g}{k+1}(\log x)^{k+1}-\twosum{n\leq x}{n\equiv a\pmod q}g(n)(\log x-\log n)\right)+O((\log x)^k).$$
We therefore get
$$(q-2-k)\twosum{m\leq x}{m\equiv a\pmod q}g(m)\log m+k\log xM_g(x)-\frac{kc_g}{k+1}(\log x)^{k+1}\ll (\log x)^k.$$
Since 
$$\log x M_g(x)-\twosum{m\leq x}{m\equiv a\pmod q}g(m)\log m=\twosum{m\leq x}{m\equiv a\pmod q}g(m)\log \frac{x}{m}=\int_1^x M_g(t)t^{-1}\,dt$$
we have
$$M_g(x)\log x-(1-\frac{k}{q-2})\int_1^x M_g(t)t^{-1}\,dt-\frac{kc_g}{(k+1)(q-2)}(\log x)^{k+1}\ll (\log x)^k.$$
We therefore conclude that for $x\geq 2$
$$M_g(x)\log x-(1-\frac{k}{q-2})\int_2^x M_g(t)t^{-1}\,dt-\frac{kc_g}{(k+1)(q-2)}(\log x)^{k+1}\ll (\log x)^k.$$
Let 
$$l=1-\frac{k}{q-2}$$
so that this is 
$$M_g(x)\log x-l\int_2^x M_g(t)t^{-1}\,dt-\frac{kc_g}{(k+1)(q-2)}(\log x)^{k+1}\ll (\log x)^k.$$
Dividing by $x(\log x)^{l+1}$ we then get 
$$x^{-1}(\log x)^{-l}\left(M_g(x)-l(\log x)^{-1}\int_2^x M_g(t)t^{-1}\,dt-\frac{kc_g}{(k+1)(q-2)}(\log x)^{k}\right)$$
$$\ll x^{-1}(\log x)^{k-l-1}.$$
We integrate this from $2$ to $x$, replacing $x$ by $t$ and $t$ by $u$.  For any $\epsilon>0$ the RHS will be 
$$\int_2^xt^{-1}(\log t)^{k-l-1}\,dt\ll_\epsilon 1+(\log x)^{k-l+\epsilon},$$
and the LHS will be
$$\int_2^x t^{-1}(\log t)^{-l}\left(M_g(t)-l(\log t)^{-1}\int_2^t M_g(u)u^{-1}\,du-\frac{kc_g}{(k+1)(q-2)}(\log t)^{k}\right)\,dt.$$ 
Re-ordering the double integral we see that 
$$(\log x)^{-l}\int_2^xM_g(t)t^{-1}\,dt-\frac{kc_g}{(k+1)(q-2)}\int_2^x t^{-1}(\log t)^{k-l}\,dt\ll 1+(\log x)^{k-l+\epsilon}.$$
However 
$$M_g(x)\log x-\frac{kc_g}{(k+1)(q-2)}(\log x)^{k+1}+O((\log x)^k)=l\int_2^x M_g(t)t^{-1}\,dt$$
so this simplifies to 
\begin{eqnarray*}
M_g(x)&=&\frac{kc_g}{(k+1)(q-2)}\left((\log x)^k+l(\log x)^{l-1}\int_2^x t^{-1}(\log t)^{k-l}\,dt\right)\\
&&\hspace{2cm}\mbox{}+O((\log x)^{l-1}+(\log x)^{k-1+\epsilon}).
\end{eqnarray*}
We know that 
$$k-l=k-1+\frac{k}{q-2}>-1$$
so
\begin{eqnarray*}
M_g(x)&=&kc_g\frac{1+l(k-l+1)^{-1}}{(k+1)(q-2)}(\log x)^k+o((\log x)^k)\\
&=&kc_g\frac{(kq-2k+q-2)/(kq-k)}{(k+1)(q-2)}(\log x)^k+o((\log x)^k)\\
&=&\frac{c_g}{q-1}(\log x)^k+o((\log x)^k).\\
\end{eqnarray*}
\end{proof}

\section{The Sieve}

\subsection{The Sieve Decomposition}

Let $\ca=(a_n)$ be the sequence given by 
$$a_n=\twosum{(a,b)\in (0,N]^2}{C(a,b),bf(a,b)=n}1.$$
We will sieve $\ca$ by the set of primes
$$\p=\{p:p\notin \p_1,p\equiv 1\pmod q\}.$$
Let 
$$x=\max\{f(a,b):(a,b)\in (0,N]^2\}=(c+o(1))N^3,$$
for some constant $c$ which depends on $f$.  
We wish to prove a positive lower bound for the sifting function 
$$S(\ca,\p,x)=\sum_{(n;P(x))=1}a_n$$
where 
$$P(z)=\twoprod{p<z}{p\in\p}p.$$
Applying the Buchstab identity we get, for some $\alpha\in (\frac12,1)$, that 
$$S(\ca,\p,x)=S(\ca,\p,N^\alpha)-\twosum{N^\alpha\leq p<x}{p\in\p}S(\ca_p,\p,p).$$
If a prime $p$ divides $bf(a,b)$ then either $p|b$ or $p|f(a,b)$.  We may therefore write 
$$S(\ca,\p,x)\geq S(\ca,\p,N^\alpha)-\twosum{N^\alpha\leq p<x}{p\in\p}(S(\ca^{(1)}_p,\p,p)+S(\ca^{(2)}_p,\p,p))$$
where $\ca^{(1)}_p$ is the subsequence of $\ca_p$ coming from pairs $(a,b)$ with $p|b$ whereas $\ca^{(2)}_p$ is the subsequence coming from $p|f(a,b)$.

If $p|b$ we must have $p\leq N$ so we can truncate the sum over $\ca^{(1)}_p$ to $p\leq N$.  As our level of distribution, Lemma \ref{mainlod}, is only nontrivial for $D_1D_2\leq N^2$ we split the sum over $\ca^{(2)}_p$ at $N^\beta$ for some $\beta\in(\frac32,2)$.  We conclude that 
$$S(\ca,\p,x)\geq S_1-S_2-S_3-S_4$$
where 
$$S_1=S(\ca,\p,N^\alpha),$$
$$S_2=\twosum{N^\alpha\leq p\leq N}{p\in\p}S(\ca^{(1)}_p,\p,p),$$
$$S_3=\twosum{N^\alpha\leq p<N^\beta}{p\in\p}S(\ca^{(2)}_p,\p,p)$$
and 
$$S_4=\twosum{N^\beta\leq p<x}{p\in\p}S(\ca^{(2)}_p,\p,p).$$
We then need a lower bound for $S_1$ and upper bounds for $S_2,S_3$ and $S_4$.  All our bounds will eventually depend on the $\beta$-sieve as given by Friedlander and Iwaniec in \cite[Theorem 11.13]{opera}.  We let $A_1,B_1$ denote the constants $A,B$ in the sieve of dimension $\frac{2}{q-1}$ and $A_2,B_2$ those for the sieve of dimension $\frac{q-2}{q-1}$.  These are the only sieves we will use.  

Throughout this section $q,f$ and $\Delta$ are fixed.  All use of the notation $o$ is as $N\rightarrow \infty$.

\subsection{The Sum $S_1$}

We have 
$$S_1=S(\ca,\p,N^\alpha).$$
 Since $\alpha<1$ we can take $D_1=1$ and $D_2=N^\alpha$ in Lemma \ref{mainlod}.  This shows that we can apply a lower bound sieve of level $N^\alpha$ as the remainder term is $O(N^{2-\delta})$ for some $\delta>0$.  Using the notation of \cite{opera} we have 
$$X=\frac{N^2}{\Delta^2}$$
and $g(p)$ is the multiplicative function given by 
$$g(p)=\begin{cases}
\frac{\rho_2(p)}{p^2} & p\in\p\\
0 & \text{Otherwise}.\\
\end{cases}$$
It follows from Lemma \ref{rhoprod} that the sieve dimension is $\frac{2}{q-1}$.  If $q>5$ then
$$\frac{2}{q-1}<\frac12,$$
and so the sifting limit is $1$.  We may therefore use the lower bound sieve to deduce that 
$$S_1\geq \frac{(B_1+o(1))N^2}{\Delta^2}\twoprod{p<N^\alpha}{p\in\p}(1-\frac{\rho_2(p)}{p^2})+O(N^{2-\delta}).$$
Applying Lemma \ref{rhoprod} we conclude that 
$$S_1\geq \frac{(c_2(f,q)B_1+o(1))N^2}{\Delta^2(\alpha \log N)^{\frac{2}{q-1}}}.$$
Finally, for any $\epsilon>0$ we can choose $\alpha$ sufficiently close to $1$ in terms of $\epsilon$ to get the bound 
$$S_1\geq \frac{(c_2(f,q)B_1-\epsilon+o(1)N^2}{\Delta^2(\log N)^{\frac{2}{q-1}}}.$$

\subsection{The Sum $S_2$}

In our bound for $S_2$ we will exploit the fact that $\alpha$ may be taken as close to $1$ as we require.  We therefore do not need to give a bound which is as sharp as possible.  It is enough to show that for any $\epsilon>0$ we can choose $\alpha<1$ depending on $\epsilon$ such that 
$$S_2\leq \frac{(\epsilon+o(1))N^2}{(\log N)^{\frac{2}{q-1}}}.$$

We have 
$$S_2\leq \twosum{N^\alpha\leq p\leq N}{p\in\p}S(\ca^{(1)}_p,\p,N^\alpha).$$
For each pair $(a,b)$ counted by $S_2$ we may write $b=pr$ where
$$p\in\p\cap [N^\alpha,N],$$
$$r\leq N^{1-\alpha}=R$$
and 
$$(r;P(R))=1.$$
In addition we have
$b\equiv b_0\pmod \Delta$.  
By our construction of $b_0$ we know that for each $p'|\Delta$ there exists an $l$ for which $p'^l|\Delta$ and 
$$b_0\not\equiv 0\pmod{p'^l}.$$
It follows that for each such prime its power dividing $b$ is the same as that dividing $b_0$.  For $p\in\p$ we have $(\Delta;p)=1$.  It follows that for each $p'|\Delta$ the power of $p'$ dividing $r$ is precisely that dividing $b_0$.  In other words we may write 
$$r=(b_0;\Delta)r'\text{ with }(r';\Delta)=1.$$ 
Given such an $r$ and a pair $(a,b)$ satisfying $C(a,b)$ the condition $r|b$ is equivalent to $r'|b$.

We may therefore write 
$$S_2\leq \twosum{r\leq R/(b_0;\Delta)}{(r;P(R)\Delta)=1}S_2(r)$$
where 
$$S_2(r)=\#\{(a,b)\in (0,N]^2:C(a,b),r|b,b/r(b_0;\Delta)\in\p,(bf(a,b);P(N^\alpha))=1\}.$$
Note that the variable of summation, $r$, is $r'$ in the above notation.

Recall that
$$\p'=\{p:p\notin\p_1,p\not\equiv 1\pmod q\}$$
and let
$$P'(z)=\twoprod{p<z}{p\in\p'}p.$$
If we let $z=N^\delta$ for some $\delta>0$ then provided $\delta<\alpha$ we have 
$$S_2(r)\leq \#\{(a,b)\in (0,N]^2:C(a,b),r|b,(bf(a,b);P(z))=(b/r;P'(z))=1\}.$$
Suppose that $\mu_1^+,\mu_2^+$ are upper bound sieves of level $z$. We have 
$$S_2(r)\leq \twosum{(a,b)\in(0,N]^2}{r|b,C(a,b)}(\sum_{d\mid\, (P(z);bf(a,b))}\mu_1^+(d))(\sum_{e\mid\, (b/r;P'(z))}\mu_2^+(e)).$$ 
Reordering the summations we get 
\begin{eqnarray*}
S_2(r)&\leq &\sum_{d|P(z)}\sum_{e|P'(z)}\mu_1^+(d)\mu_2^+(e)\#\{(a,b)\in (0,N]^2:C(a,b),re|b,d|bf(a,b)\}\\
&=&\sum_{d|P(z)}\sum_{e|P'(z)}\mu_1^+(d)\mu_2^+(e)R_1(d,re).\\
\end{eqnarray*}
If $\delta$ is sufficiently small so that 
$$1-\alpha+2\delta<1$$
then 
$$rde\leq N.$$
Furthermore $(re;d)=(dre;\Delta)=1$ so Lemma \ref{fliplod} applies and we get 
$$S_2(r)\leq\sum_{d|P(z)}\sum_{e|P'(z)}\mu_1^+(d)\mu_2^+(e)\left(\frac{N^2\rho_2(d)}{d^2re\Delta^2}+O_\epsilon(Nd^\epsilon)\right).$$ 
The contribution of the error term to $S_2$ is bounded by 
$$\sum_{r\leq R}N^{1+2\delta+\epsilon}\ll N^{2-\alpha+2\delta+\epsilon}=o(\frac{N^2}{(\log N)^{\frac{2}{q-1}}}),$$
in view of our assumption on the size of $\delta$.  

The main term in the above estimate for $S_2(r)$ is 
$$\frac{N^2}{r\Delta^2}\left(\sum_{d|P(z)}\mu_1^+(d)\frac{\rho_2(d)}{d^2}\right)\left(\sum_{e|P'(z)}\frac{\mu_2^+(e)}{e}\right).$$
The two sums may now be estimated using the sieve.  We let $(\mu_1^+)$ be a sieve of dimension $\frac{2}{q-1}$.  It follows using Lemma \ref{rhoprod} that 
\begin{eqnarray*}
\sum_{d|P(z)}\mu_1^+(d)\frac{\rho_2(d)}{d^2}&\leq&(A_1+o(1))\twoprod{p<z}{p\in\p}(1-\frac{\rho_2(p)}{p^2})\\
&=&(A_1+o(1))\frac{c_2(f,q)}{(\delta\log N)^{\frac{2}{q-1}}}.\\
\end{eqnarray*} 
We let $(\mu_2^+)$ be a sieve of dimension $\frac{q-2}{q-1}$ and thus we get 
$$\sum_{e|P'(z)}\frac{\mu_2^+(e)}{e}\leq (A_2+o(1))\twoprod{p<z}{p\in\p'}(1-\frac{1}{p}).$$
Finally we have the bound 
$$\twosum{r\leq R/(b_0;\Delta)}{(r;P(R)\Delta=1}\frac{1}{r}\leq \twoprod{p<R}{p\in\p'}(1-\frac{1}{p})^{-1}.$$
By taking $\alpha$ sufficiently close to $1$ we can assume that $R<z$.  It follows that 
$$\twoprod{p<z}{p\in\p'}(1-\frac{1}{p})\twoprod{p<R}{p\in\p'}(1-\frac{1}{p})^{-1}=\twoprod{R<p<z}{p\in\p'}(1-\frac{1}{p})\sim (\frac{1-\alpha}{\delta})^{\frac{q-2}{q-1}}.$$
We finally conclude that 
$$S_2\leq \left(\frac{1-\alpha}{\delta}\right)^{\frac{q-2}{q-1}}\frac{(A_1A_2c_2(f,q)+o(1))N^2}{\Delta^2(\delta\log N)^{\frac{2}{q-1}}}.$$
It follows that for any $\epsilon>0$ there exists an $\alpha<1$ depending on $\epsilon$ such that 
$$S_2\leq\frac{(\epsilon+o(1))N^2}{(\log N)^{\frac{2}{q-1}}}.$$

\subsection{The Sum $S_3$}

Let $S_3(P_1,P_2)$ denote the part of $S_3$ with $P_1\leq p<P_2$.  We have 
$$S_3(P_1,P_2)\leq \twosum{P_1\leq p<P_2}{p\in\p}S(\ca^{(2)}_p,\p,P_1).$$
We will apply an upper bound sieve to each summand separately.  For each prime $p$ and each $d\in\N$ we have 
$$\sum_{n\equiv 0\pmod d}(a^{(2)}_p)_n=R(p,d).$$
If $d<p$ then clearly $(d;p)=1$.  We will apply Lemma \ref{mainlod} with $D_1=P_2$ and 
$$D_2=D_2(P_1,P_2)=\min(N^{1-\gamma},P_1,\frac{N^{2-\gamma}}{P_2})$$
for some $\gamma>0$ which we will choose arbitrarily small.  We then have 
$$\twosum{p\leq P_2,d\leq D_2}{(d;\Delta)=1,\mu(d)^2=1}|R(p,d)-\frac{\rho_1(p)\rho_2(d)N^2}{p^2d^2\Delta^2}|\ll N^{2-\delta},$$
for some $\delta>0$ which depends on $\gamma$.

Applying the upper bound sieve of dimension $\frac{2}{q-1}$ results in 
$$S_3(P_1,P_2)\leq \frac{(A_1+o(1))N^2}{\Delta^2}\left(\twoprod{p<D_2}{p\in\p}(1-\frac{\rho_2(p)}{p^2})\right)\twosum{P_1\leq p<P_2}{p\in\p}\frac{\rho_1(p)}{p^2}+O(N^{2-\delta}).$$
We can evaluate the product using Lemma \ref{rhoprod} to get 
$$S_3(P_1,P_2)\leq \frac{(c_2(f,q)A_1+o(1))N^2}{\Delta^2(\log D_2)^{\frac{2}{q-1}}}\twosum{P_1\leq p<P_2}{p\in\p}\frac{\rho_1(p)}{p^2}+O(N^{2-\delta}).$$
In addition, using our previous convention that the value  $c$ may vary from line to line, we have 
\begin{eqnarray*}
\twosum{p\leq x}{p\in\p}\frac{\rho_1(p)}{p^2}&=&\twosum{p\leq x}{p\in\p}\frac{(p-1)\nu(p)+1}{p^2}\\
&=&\twosum{p\leq x}{p\in\p}\frac{\nu(p)}{p}+c+o(1)\\
&=&\frac{1}{q-1}\log\log x+c+o(1).\\
\end{eqnarray*}
We first consider $S_3(N^\alpha,N^{2-\alpha})$.  As in the previous section we will take $\alpha$ close to $1$ which is enough to make this part of the sum small.  We take $D_2=N^\alpha$, getting 
$$S_3(N^\alpha,N^{2-\alpha})\leq \frac{1}{q-1}(\log(2-\alpha)-\log \alpha)\frac{(c_2(f,q)A_1+o(1))N^2}{\Delta^2(\alpha\log N)^{\frac{2}{q-1}}}.$$
It follows that for any $\epsilon>0$ we can choose $\alpha$ sufficiently close to $1$ to deduce that 
$$S_3(N^\alpha,N^{2-\alpha})\leq \frac{(\epsilon+o(1))N^2}{(\log N)^{\frac{2}{q-1}}}.$$

It remains to estimate $S_3(N^{2-\alpha},N^\beta)$. We divide  this range into dyadic intervals $[P_1,2P_1)$.  For each such interval we have
$$D_2=\frac{N^{2-\gamma}}{P_1}.$$
By taking $\gamma<2-\beta$ we have $D_2\geq 1$ for all the dyadic intervals.  In addition if $p\sim P_1$ then 
$$(\log D_2)^{-\frac{2}{q-1}}=(\log\frac{N^{2-\gamma}}{P_1})^{-\frac{2}{q-1}}\leq (\log\frac{N^{2-\gamma}}{p})^{-\frac{2}{q-1}}.$$
We therefore have 
$$S_3(P_1,2P_1)\leq \frac{(c_2(f,q)A_1+o(1))N^2}{\Delta^2}\twosum{P_1\leq p<2P_1}{p\in\p}\frac{\rho_1(p)}{p^2(\log\frac{N^{2-\gamma}}{p})^{\frac{2}{q-1}}}+O(N^{2-\delta})$$
and thus 
$$S_3(N^{2-\alpha},N^\beta)\leq \frac{(c_2(f,q)A_1+o(1))N^2}{\Delta^2(\log N)^{\frac{2}{q-1}}}\twosum{N^{2-\alpha}\leq p<N^\beta}{p\in\p}\frac{\rho_1(p)}{p^2(2-\gamma-\log p/\log N)^{\frac{2}{q-1}}}.$$
We have 
$$\twosum{N\leq p<t}{p\in \p}\frac{\rho_1(p)}{p^2}=\frac{1}{q-1}(\log\log t-\log\log N)+o(1),$$
so we can sum by parts to get 
\begin{eqnarray*}
\lefteqn{\twosum{N^{2-\alpha}\leq p<N^\beta}{p\in\p}\frac{\rho_1(p)}{p^2(2-\gamma-\log p/\log N)^{\frac{2}{q-1}}}}\\
&\leq&\twosum{N\leq p<N^\beta}{p\in\p}\frac{\rho_1(p)}{p^2(2-\gamma-\log p/\log N)^{\frac{2}{q-1}}}\\
&=&\frac{1}{q-1}\frac{\log\beta}{(2-\gamma-\beta)^{\frac{2}{q-1}}}-\frac{2}{(q-1)^2}\int_N^{N^\beta}\frac{\log\log t-\log\log N}{t\log N(2-\gamma-\log t/\log N))^{\frac{q+1}{q-1}}}\,dt+o(1)\\
&=&\frac{1}{q-1}\frac{\log\beta}{(2-\gamma-\beta)^{\frac{2}{q-1}}}-\frac{2}{(q-1)^2}\int_1^{\beta}\log s(2-\gamma-s)^{-\frac{q+1}{q-1}}\,ds+o(1)\\
&=&\frac{1}{q-1}\int_1^\beta(2-\gamma-s)^{-\frac{2}{q-1}}\,\frac{ds}{s}+o(1).\\
\end{eqnarray*}
We conclude that 
$$S_3(N^{2-\alpha},N^\beta)\leq \frac{(c_2(f,q)A_1+o(1))N^2}{\Delta^2(q-1)(\log N)^{\frac{2}{q-1}}}\int_1^\beta(2-\gamma-s)^{-\frac{2}{q-1}}\,\frac{ds}{s}.$$
Combining the above bounds we see that for any $\epsilon>0$ we can choose $\alpha$ sufficiently close to $1$ and $\gamma$ sufficiently small to get the bound 
$$S_3\leq \frac{(c_2(f,q)A_1+\epsilon+o(1))N^2}{\Delta^2(q-1)(\log N)^{\frac{2}{q-1}}}\int_1^\beta(2-s)^{-\frac{2}{q-1}}\,\frac{ds}{s}.$$

\subsection{The Sum $S_4$}

We have
$$S_4\leq\twosum{N^\beta\leq p<x}{ p\in\p}S(\ca^{(2)}_p,\p,N^\beta).$$
For each pair $(a,b)$ counted by $S_4$ we can write $f(a,b)=pr$ where 
$$p\in [N^\beta,x]\cap \p,$$
$$r\leq \frac{x}{N^\beta}=R$$
and 
$$(r;P(R))=1.$$
Let $f_0=f(a_0,b_0)$.  For each prime $p'|\Delta$ we know that there exists an $l$ for which $p'^l|\Delta$ and 
$$f_0\not\equiv0\pmod{p'^l}.$$
It follows that the power of $p'$ dividing $f(a,b)$ is the same as that dividing $f_0$.  Since $(p;\Delta)=1$ this power is the same as that dividing $r$.  In other words we can write 
$$r=(f_0;\Delta)r',\quad (r';\Delta)=1.$$
Given a pair $(a,b)$ with $C(a,b)$ the condition $r|f(a,b)$ is equivalent to $r'|f(a,b)$.

The prime $q$ divides $\Delta$.  In addition since $p\in\p$ we have $p\equiv 1\pmod q$.  It follows that there exists $r_0$ depending only on $a_0,b_0,\Delta$ with $(r_0;q)=1$ such that 
$$r'\equiv r_0\pmod q.$$

We may now write 
$$S_4\leq \twosum{r\leq R/(f_0;\Delta)}{(r;P(R)\Delta)=1,r\equiv r_0\pmod q}S_4(r)$$
where
$$S_4(r)=\#\{(a,b)\in (0,N]^2:C(a,b),r|f(a,b),(f(a,b)/r;P'(z'))=(bf(a,b);P(z))=1\},$$
for some $z,z'$ satisfying $0\leq z,z'\leq N^\beta$.    Note that the variable of summation, $r$, is $r'$ in the above notation.  We will split the sum over $r$ into dyadic segments $R_1\leq r<2R_1$.

Let $\mu_1^+,\mu_2^+$ be upper bound sieves of levels  $D'$ and $D$, respectively, where $D,D'$ depend on $R_1$.  It follows that 
\begin{eqnarray*}
S_4(r)&\leq&\twosum{d|P'(z')}{(d;r)=1}\sum_{e|P(z)}\mu_1^+(d)\mu_2^+(e)\#\{(a,b)\in (0,N]^2:C(a,b),dr|f(a,b),e|bf(a,b)\}\\
&=&\twosum{d|P'(z')}{(d;r)=1}\sum_{e|P(z)}\mu_1^+(d)\mu_2^+(e)R(dr,e).\\
\end{eqnarray*}
Since $(dr;e)=(dre;\Delta)=1$ we may apply Lemma \ref{mainlod}.  If we write $D=N^\eta,D'=N^{\eta'}$ this requires that 
$$\eta\leq 1-\delta$$
and 
$$\eta+\eta'+\frac{\log R_1}{\log N}\leq 2-\delta,$$
for some $\delta>0$ which we will eventually take arbitrarily small.

Given these assumptions on $\eta$ and $\eta'$ the contribution of the error term to $S_4$ is $o(\frac{N^2}{\log N^{\frac{2}{q-1}}})$.  

It remains to deal with the main term coming from Lemma \ref{mainlod}.  This is 
$$\frac{N^2\rho_1(r)}{r^2\Delta^2}\left(\twosum{d|\,P'(z')}{(d;r)=1}\mu_1^+(d)\frac{\rho_1(d)}{d^2}\right)\left(\sum_{e|\,P(z)}\mu_2^+(e)\frac{\rho_2(e)}{e^2}\right).$$
The two terms can now be estimated using the sieve.  Considering the results of Lemma \ref{rhoprod} we let $\mu_1^+$ be a sieve of dimension $\frac{q-2}{q-1}$ and we let $\mu_2^+$ be a sieve of dimension $\frac{2}{q-1}$.  We may assume that $z=D$ and either $z'=D'$ or $z'=N^\beta\leq D'\leq N^2$.  It follows that the values of $z,z'$ do not affect the sieve upper bounds and therefore 
\begin{eqnarray*}
\twosum{d|P'(z')}{(d;r)=1}\mu_1^+(d)\frac{\rho_1(d)}{d^2}&\leq& (A_2+o(1))\twoprod{p<D'}{p\in\p'}(1-\frac{\rho_1(p)}{p^2})\prod_{p|r}(1-\frac{\rho_1(p)}{p^2})^{-1}\\
&=&\frac{(c_1(f,q)A_2+o(1))}{(\eta'\log N)^{\frac{q-2}{q-1}}}\prod_{p|r}(1-\frac{\rho_1(p)}{p^2})^{-1}\\
\end{eqnarray*}
and
$$\sum_{e|P(z)}\mu_2^+(e)\frac{\rho_2(e)}{e^2}\leq(A_1+o(1))\twoprod{p<z}{p\in\p}(1-\frac{\rho_2(p)}{p^2})=\frac{(c_2(f,q)A_1+o(1))}{(\eta\log N)^{\frac{2}{q-1}}}.$$

The contribution to our upper bound from the $\eta,\eta'$ is then 
$$\frac{1}{\eta'^{\frac{q-2}{q-1}}\eta^{\frac{2}{q-1}}}.$$
Therefore, to give an optimal result,  we want to maximise 
$$\eta'^{q-2}\eta^2$$
subject to the constraints 
$$\eta\leq 1-\delta$$
and 
$$\eta+\eta'\leq2-\frac{\log R_1}{\log N}-\delta.$$
By monotonicity it is clear that the maximum occurs when we have equality in the last constraint so 
$$\eta'=2-\delta-\eta-\frac{\log R_1}{\log N}.$$
We therefore wish to maximise 
$$(2-\delta-\eta-\frac{\log R_1}{\log N})^{q-2}\eta^2$$
for $\eta\in (0,1-\delta]$.  Taking logs we maximise 
$$(q-2)\log(2-\delta-\eta-\frac{\log R_1}{\log N})+2\log\eta$$
so we solve 
$$-(q-2)(2-\delta-\eta-\frac{\log R_1}{\log N})^{-1}+2\eta^{-1}=0.$$
This gives 
$$\eta=2q^{-1}(2-\delta-\frac{\log R_1}{\log N}).$$
Observe that this is in $(0,1-\delta]$ if $q\geq 5$ and $\delta$ is sufficiently small.  We then get 
$$\eta'=2-\delta-\eta-\frac{\log R_1}{\log N}=(1-2q^{-1})(2-\delta-\frac{\log R_1}{\log N}).$$
If $q\geq 5$ and $\delta$ is sufficiently small then $\eta'>0$.  The factor coming from $\eta,\eta'$ is thus 
$$\left((1-2q^{-1})(2-\delta-\frac{\log R_1}{\log N})\right)^{-\frac{q-2}{q-1}}\left(2q^{-1}(2-\delta-\frac{\log R_1}{\log N})\right)^{-\frac{2}{q-1}}.$$
This increases as we increase $R_1$ so we can replace $R_1$ by $r$ getting the smooth weight 
$$w(r,\delta)=(1-2q^{-1})^{-\frac{q-2}{q-1}}(2q^{-1})^{-\frac{2}{q-1}}\left(2-\delta-\frac{\log R_1}{\log N}\right)^{-\frac{q}{q-1}}.$$
Combining all of the above we see that the main term in our estimate for $S_4$ is 
$$\frac{A_1A_2c_1(f,q)c_2(f,q)N^2}{\Delta^2(\log N)^{\frac{q}{q-1}}}\twosum{r\leq R/(f_0;\Delta)}{r\equiv r_0\pmod q}w(r,\delta)g(r),$$
where $g(r)$ is the multiplicative function which is $0$ unless all the prime factors of $r$ are in $\p'$, in which case it is given by 
$$g(r)=\frac{\rho_1(r)}{r^2}\prod_{p|r}(1-\frac{\rho_1(p)}{p^2})^{-1}.$$

To estimate the sum over $r$ we begin by dealing with the $r$ which are  squarefree.  

\begin{lem}
The multiplicative function $g$, when restricted to squarefree numbers, satisfies all the hypotheses of Lemma \ref{multlem}.
\end{lem}

\begin{proof}
Since $q|\Delta$ we know that $g(q)=0$.  In addition if $g(p)\ne 0$ then $p\in\p'$ so $p\not\equiv 1\pmod q$.

If $a\not\equiv 0,1\pmod q$ then
\begin{eqnarray*}
\twosum{p\leq x}{p\equiv a\pmod q}g(p)\log p&=&\twosum{p\leq x}{p\equiv a\pmod q}(1-\frac{\rho_1(p)}{p^2})^{-1}\frac{\rho_1(p)}{p^2}\log p+O(1)\\
&=&\twosum{p\leq x}{p\equiv a\pmod q}\frac{\rho_1(p)}{p^2}\log p+O(1)\\
&=&\twosum{p\leq x}{p\equiv a\pmod q}\frac{(p-1)\nu(p)+1}{p^2}\log p+O(1)\\
&=&\frac{1}{q-1}\log x+O(1).\\
\end{eqnarray*}
This establishes (\ref{mult1}) and (\ref{mult5}) with 
$$k=\frac{q-2}{q-1}.$$
If $2\leq w<z$ then, by Lemma \ref{rhoprod},  we have 
$$\prod_{w\leq p<z}(1+g(p))=\twoprod{w\leq p<z}{p\in\p'}(1-\frac{\rho_1(p)}{p^2})^{-1}\ll (\frac{\log z}{\log w})^{\frac{q-2}{q-1}},$$
so (\ref{mult2}) holds.  Finally 
$$\sum_p g(p)^2\log p= \sum_{p\in\p'} (1-\frac{\rho_1(p)}{p^2})^{-2}\frac{\rho_1(p)^2}{p^4}\log p\ll \sum_p\frac{1}{p^2}\log p<\infty$$
and therefore (\ref{mult3}) also holds.
\end{proof}

Summing by parts and applying Lemma \ref{multlem} we have 
\begin{eqnarray*}
\lefteqn{\twosum{r\leq R/(f_0;\Delta)}{\mu(r)\ne 0,\,r\equiv r_0\!\!\pmod q}w(r,\delta)g(r)}\\
&\leq&\twosum{r\leq R}{\mu(r)\ne 0,\,r\equiv r_0\!\!\pmod q}w(r,\delta)g(r)\\
&=&w(R,\delta)\twosum{r\leq R}{\mu(r)\ne 0,\,r\equiv r_0\!\!\pmod q}g(r)-\int_1^R \left(\twosum{r\leq t}{\mu(r)\ne 0,\,r\equiv r_0\!\!\pmod q}g(r)\right)w'(t)\,dt\\
&=&(c_g+o(1))\frac{1}{q-1}\left(w(R,\delta)(\log R)^{\frac{q-2}{q-1}}-\int_1^R w'(t)(\log t)^{\frac{q-2}{q-1}}\,dt\right)\\
&=&(c_g+o(1))\frac{q-2}{(q-1)^2}\int_1^R w(t,\delta)(\log t)^{\frac{-1}{q-1}}t^{-1}\,dt\\
&=&(c_g+o(1))\frac{q-2}{(q-1)^2}(\log N)^{\frac{q-2}{q-1}}\int_0^{\log R/\log N} w(N^s,\delta)s^{\frac{-1}{q-1}}\,ds,\\
\end{eqnarray*}
where 
$$c_g=\frac{1}{\Gamma(2-\frac{1}{q-1})}\prod_p(1-\frac{1}{p})^{\frac{q-2}{q-1}}(1+g(p)).$$
Observe that 
$$w(N^s,\delta)=(1-2q^{-1})^{-\frac{q-2}{q-1}}(2q^{-1})^{-\frac{2}{q-1}}(2-\delta-s)^{-\frac{q}{q-1}}$$
does not depend on $N$.  In addition 
$$\frac{\log R}{\log N}=\frac{\log x}{\log N}-\beta=3-\beta+o(1)$$
as $N\rightarrow \infty$.  We may therefore replace the upper limit of integration by $3-\beta$ at the cost of an error which is $o(1)$.  We conclude that 
$$\twosum{r\leq R/(f_0;\Delta)}{\mu(r)\ne 0,\,r\equiv r_0\!\!\pmod q}w(r,\delta)g(r)\leq (c_g+o(1))\frac{q-2}{(q-1)^2}(\log N)^{\frac{q-2}{q-1}}\int_0^{3-\beta} w(N^s,\delta)s^{\frac{-1}{q-1}}\,ds.$$

Let 
$$W(s)=w(N^s,0)s^{\frac{-1}{q-1}}.$$
For any $\epsilon>0$ we can choose a sufficiently small $\delta$ to get 
$$\int_0^{3-\beta} w(N^s,\delta)s^{\frac{-1}{q-1}}\,ds\leq \int_0^{3-\beta} W(s)\,ds+\epsilon+o(1)$$
and thus 
$$\twosum{r\leq R/(f_0;\Delta)}{\mu(r)\ne 0,\,r\equiv r_0\!\!\!\pmod q}w(r,\delta)g(r)\leq (c_g+\epsilon+o(1))\frac{q-2}{(q-1)^2}(\log N)^{\frac{q-2}{q-1}}\int_0^{3-\beta}W(s)\,ds.$$

It remains to deal with the sum over those $r$ which are not squarefree.

\begin{lem}\label{notsquarefree}
For any $\epsilon>0$ there exists a $P_1$, depending on $\epsilon,q$ and $f$ but not on $N$,  such that if we include all primes $p\leq P_1$ in $\p_1$ then 
$$\twosum{r\leq R/(f_0;\Delta)}{\mu(r)=0,\,r\equiv r_0\!\!\!\pmod q}w(r,\delta)g(r)\leq (\epsilon+o(1))(\log N)^{\frac{q-2}{q-1}}.$$
\end{lem}

\begin{proof}
Any $r\in\N$ can be written uniquely as $r=r_1r_2$ for some squarefree $r_1$ and some  squarefull $r_2$ satisfying $(r_1;r_2)=1$.  In addition if $\mu(r)=0$ then $r_2>1$.  Since $w(r,\delta)\ll 1$ and $g(r)\geq 0$ our sum may be bounded by 
$$\twosum{r_1r_2\leq R/(f_0;\Delta)}{r_2>1}g(r_1)g(r_2),$$
where the sum is restricted to squarefree $r_1$ and squarefull $r_2$ with $(r_1;r_2)=1$.

Since $g(r)$ is supported on numbers having no prime factor in $\p_1$ we can use Lemma \ref{rho1bound} to deduce that for all $r$ 
$$g(r)\ll_\epsilon r^{-\frac23+\epsilon}.$$
It follows that 
$$\sum_{r \text{ squarefull}}g(r)<\infty.$$
Furthermore, if we include all primes up to  $P_1$ in $\p_1$ then all terms in this sum with $r\leq P_1$ are $0$.  It follows that for any $\epsilon>0$ we can choose $P_1$ sufficiently large so that 
$$\twosum{r>1}{r \text{ squarefull}}g(r)<\epsilon.$$
Our original sum may therefore be bounded by 
$$(\twosum{r_1\leq R/(f_0;\Delta)}{\mu(r_1)\ne 0}g(r_1))(\sum_{r_2\text{ squarefull}}g(r_2)).$$
Using Lemma \ref{multlem} the first sum is $O((\log R)^{\frac{q-2}{q-1}})$  so the result follows.  
\end{proof}

It follows from the last lemma that, with a suitable choice of $P_1$, the non-squarefree $r$ give a contribution to $S_4$ bounded by 
$$\frac{(\epsilon+o(1))}{(\log N)^{\frac{2}{q-1}}}.$$

Combining all of the results of this subsection we see that for any $\epsilon>0$, by taking sufficiently many small primes in $\p_1$ and $\delta$ sufficiently small, we get the bound 
$$S_4\leq\int_0^{3-\beta} W(s)\,ds\frac{(A_1A_2c_gc_1(f,q)c_2(f,q)+\epsilon+o(1))(q-2)N^2}{\Delta^2(q-1)^2(\log N)^{\frac{2}{q-1}}}.$$
Finally we must remove the constants $c_g,c_1(f,q)$ from this bound. Recall that these are defined by 
\begin{eqnarray*}
c_g&=&\frac{1}{\Gamma(2-\frac{1}{q-1})}\prod_p(1-\frac{1}{p})^{\frac{q-2}{q-1}}(1+g(p))\\
&=&\frac{1}{\Gamma(2-\frac{1}{q-1})}\lim_{x\rightarrow \infty}\prod_{p\leq x}(1-\frac{1}{p})^{\frac{q-2}{q-1}}(1+g(p))\\
&=&\frac{1}{\Gamma(2-\frac{1}{q-1})}\lim_{x\rightarrow \infty}(\frac{e^{-\gamma}}{\log x})^{\frac{q-2}{q-1}}\twoprod{p\leq x}{p\in\p'}(1-\frac{\rho_1(p)}{p^2})^{-1}\\
\end{eqnarray*}
and 
$$c_1(f,q)=\lim_{x\rightarrow\infty}(\log x)^{\frac{q-2}{q-1}}\twoprod{p<x}{p\in\p'}(1-\frac{\rho_1(p)}{p^2}).$$
It follows that 
$$c_gc_1(f,q)=\frac{e^{-\gamma\frac{q-2}{q-1}}}{\Gamma(2-\frac{1}{q-1})}.$$
We therefore conclude that 
$$S_4\leq\int_0^{3-\beta} W(s)\,ds\frac{\left(A_1A_2e^{-\gamma\frac{q-2}{q-1}}c_2(f,q)(q-2)+\epsilon+o(1)\right)N^2}{\Delta^2\Gamma(2-\frac{1}{q-1})(q-1)^2(\log N)^{\frac{2}{q-1}}}.$$

\subsection{Conclusion}

Combining the bounds for $S_1,S_2,S_3$ and $S_4$ we conclude that for any $\epsilon>0$ we can take sufficiently many small primes in $\p_1$ so that we have, as $N\rightarrow \infty$, that  
$$S(\ca,\p,x)\geq \frac{c_2(f,q)N^2}{\Delta^2(\log N)^{\frac{2}{q-1}}}(F(q)-\epsilon+o(1))$$
where 
$$F(q)=B_1-\frac{A_1}{q-1}\int_1^\beta(2-s)^{-\frac{2}{q-1}}\,\frac{ds}{s}-\frac{A_1A_2e^{-\gamma\frac{q-2}{q-1}}(q-2)}{\Gamma(2-\frac{1}{q-1})(q-1)^2}\int_0^{3-\beta} W(s)\,ds$$
and
$$W(s)=(1-2q^{-1})^{-\frac{q-2}{q-1}}(2q^{-1})^{-\frac{2}{q-1}}(2-s)^{-\frac{q}{q-1}}s^{\frac{-1}{q-1}}.$$
Recall that the values $A_1,B_1$ and $A_2$ all depend on $q$.  As the sieve dimension $\kappa\rightarrow 0$ we have $A(\kappa),B(\kappa)\rightarrow 1$.  It follows that 
$$\lim_{q\rightarrow \infty}F(q)=1.$$
Therefore $F(q)$ is positive for $q\geq q_0$ for some absolute $q_0$.  For any such $q$ we can then choose $N$ sufficiently large to get $S(\ca,\p,x)>0$ and thus (\ref{eq:diophantine}) has a rational solution.

To give the best possible bound we must choose $\beta$ to minimise
$$\int_1^\beta(2-s)^{-\frac{2}{q-1}}\,\frac{ds}{s}+\frac{A_2e^{-\gamma\frac{q-2}{q-1}}(q-2)}{\Gamma(2-\frac{1}{q-1})(q-1)}\int_0^{3-\beta} W(s)\,ds.$$
Thus we must solve 
$$(2-\beta)^{-\frac{2}{q-1}}\beta^{-1}-\frac{A_2e^{-\gamma\frac{q-2}{q-1}}(q-2)}{\Gamma(2-\frac{1}{q-1})(q-1)}W(3-\beta)=0,$$
that is 
$$(2-\beta)^{-\frac{2}{q-1}}\beta^{-1}-\frac{A_2e^{-\gamma\frac{q-2}{q-1}}(q-2)}{\Gamma(2-\frac{1}{q-1})(q-1)}(1-2q^{-1})^{-\frac{q-2}{q-1}}(2q^{-1})^{-\frac{2}{q-1}}(\beta-1)^{-\frac{q}{q-1}}(3-\beta)^{\frac{-1}{q-1}}=0.$$

To complete the proof of Theorem \ref{mainthm} we must show that for all primes $q\geq 7$ there exists a choice of $\beta\in (\frac{3}{2},2)$ for which $F(q)>0$.  The case $q=7$ is the most delicate numerically so we deal with it first.  

From Friedlander and Iwaniec's table in \cite[Section 11.19]{opera} we obtain the value 
$$A_2=A(5/6)=2.56140\ldots.$$
The constants $A_1,B_1$ are given by \cite[(11.62)]{opera}.  We find by numerical integration that 
$$A_1=A(1/3)=1.27713\ldots$$
and 
$$B1=B(1/3)=0.71213\ldots.$$
By solving the above equation numerically we discover that the optimal choice for $\beta$ is approximately $1.994$.  We conclude, evaluating all integral numerically, that 
$$F(7)\approx 0.0504>0.$$
Due to the use of numerical integration we cannot be completely sure that $F(7)>0$.  However we are confident that the computations were sufficiently accurate to make this extremely likely.  

For $q\geq 11$ we do not need to be quite so careful.  Since $A_2\leq A(1)$ we have 
$$F(q)\geq B_1-\frac{A_1}{q-1}\int_1^\beta(2-s)^{-\frac{2}{q-1}}\,\frac{ds}{s}-\frac{A_1A(1)e^{-\gamma\frac{q-2}{q-1}}(q-2)}{\Gamma(2-\frac{1}{q-1})(q-1)^2}\int_0^{3-\beta} W(s)\,ds.$$
As $q$ increases $B_1=B(\frac{2}{q-1})$ is increasing whereas $A_1=A(\frac{2}{q-1})$ is decreasing.  In addition, for any $s\in (1,\beta)$ the quantity 
$$(2-s)^{-\frac{2}{q-1}}$$
is decreasing, as are 
$$e^{-\gamma\frac{q-2}{q-1}},$$
$$\frac{q-2}{(q-1)^2}$$
and 
$$\frac{1}{\Gamma(2-\frac{1}{q-1})}.$$
Recall that 
$$W(s)=(1-2q^{-1})^{-\frac{q-2}{q-1}}(2q^{-1})^{-\frac{2}{q-1}}(2-s)^{-\frac{q}{q-1}}s^{\frac{-1}{q-1}}.$$
It can be shown that for any $s\in (0,3-\beta)$ this decreases as we increase $q$.  

We can conclude that, for a fixed $\beta\in (\frac{3}{2},2)$, the above bound for $F(q)$ is an increasing function of $q$.  It follows that it is sufficient that the bound is positive when $q=11$.  Using that 
$$A(1)=2e^\gamma=3.562144\ldots,$$
$$A(1/5)=1.15147\ldots$$
and 
$$B(1/5)=0.92055\ldots$$
we can deduce, by taking $\beta=1.9$, that for any prime $q\geq 11$ we have 
$$F(q)\geq 0.514.$$

In conclusion, $F(q)>0$  for all primes $q\geq 7$ so Theorem \ref{mainthm} holds for all primes $q\geq 7$.  

\addcontentsline{toc}{section}{References} 
\bibliographystyle{plain}
\bibliography{../biblio}

\bigskip
\bigskip

Mathematical Institute,

University of Oxford,

Andrew Wiles Building, 

Radcliffe Observatory Quarter, 

Woodstock Road, 

Oxford 

OX2 6GG 

UK
\bigskip

{\tt irving@maths.ox.ac.uk}
 
\end{document}